\documentclass[reqno,12pt,a4paper]{amsart}

\usepackage[latin1]{inputenc}

\usepackage{enumerate}
\usepackage{amstext,amsmath,amsthm,amsfonts,amssymb,amscd}
\usepackage{latexsym,mathrsfs,dsfont,euscript}
\usepackage{txfonts}
\usepackage{color}

\usepackage{hyperref} 

\newtheorem{theorem}{Theorem}
\newtheorem{proposition}[theorem]{Proposition}

\theoremstyle{definition}

\theoremstyle{remark}

\numberwithin{equation}{section}

\newcommand{\abs}[1]{\left\vert#1\right\vert}
\newcommand{\set}[1]{\left\{#1\right\}}

\allowdisplaybreaks
\begin{document}
\title[]{Uniform spaces and the Newtonian structure of (big)data affinity kernels}
%
%


\author[]{Hugo Aimar}
\email{haimar@santafe-conicet.gov.ar}
\author[]{Ivana G\'{o}mez}
\email{ivanagomez@santafe-conicet.gov.ar}
%

\begin{abstract}
Let $X$ be a (data) set. Let $K(x,y)>0$ be a measure of the affinity between the data points $x$ and $y$.
We prove that $K$ has the structure of a Newtonian potential $K(x,y)=\varphi(d(x,y))$ with $\varphi$ decreasing
and $d$ a quasi-metric on $X$ under two mild conditions on $K$. The first is that the affinity of each $x$ to itself
is infinite and that for $x\neq y$ the affinity is positive and finite. The second is a quantitative transitivity; if
the affinity between $x$ and $y$ is larger than $\lambda>0$ and the affinity of $y$ and $z$ is also larger than $\lambda$,
then the affinity between $x$ and $z$ is larger than $\nu(\lambda)$. The function $\nu$ is concave, increasing, continuous
from $\mathbb{R}^+$ onto  $\mathbb{R}^+$ with $\nu(\lambda)<\lambda$ for every $\lambda>0$.
\end{abstract}
\maketitle
\section{Introduction}
Isaac Newton in the seventeenth century started the endless quantitative approach to the understanding of nature. The quantitative character of the formulation of the Law of Universal Gravitation, should not hide its deep qualitative aspects. Now, more than 300 years later, we are able to explain the central fields as gradients of radial potentials centered at the ``object'' generating the field. Usually the potential take the form $V(x)=\varphi(\abs{x})$ where $\varphi$ is a decreasing profile and $\abs{x}$ is the distance of the point $x$, where the field is to be measured, to the origin  of coordinates supporting the mass or the charge that generates the field. In the Euclidean $n$-dimensional space the profile $\varphi(r)=r^{-n+2}$ gives the fundamental solution of the Laplacian. And the kernel $K(x,y)=\varphi(\abs{x-y})$ provides the basic information in order to produce the continuous models by convolution with the mass or charge densities that determine the system. These facts are also the starting points for harmonic analysis.

Our aim in this paper is to use arguments and results strongly related to the theory of uniform spaces, in order to give sufficient conditions on a kernel function $K(x,y)$ defined on an abstract setting, in such a way that $K(x,y)=\varphi(d(x,y))$ with $\varphi$ a decreasing profile and $d$ a (quasi)metric on $X$.

In other words, we aim to obtain an abstract form of Newtonian potentials for general kernels. Our approach in the search of conditions on the kernel $K(x,y)$ will be based in the heuristic associated to the interpretation of $K(x,y)$ as an affinity matrix for (big) data. Amazingly enough the abstract of the paper \cite{CoifmanPNAS2005} of the Coifman Group leading the harmonic analysis approach to determine the structure of big data sets, reads
\textit{``The process of iterating or diffusing the Markov matrix is seen as a generalization of some aspects of the Newtonian paradigm, in which local infinitesimal transitions of a system lead to global macroscopic by integration.''}
In some  sense the result of this paper shows that also the potential theoretic Newtonian view of nature is still close to these problems. By the way, our approach could be a good example of how qualitative aspects of a system leads to structural results of the model.

Let $X$ be a set (data set). Each element $x$ of $X$ is understood as a data point. Let $K(x,y)$ be a nonnegative number measuring the affinity bet\-ween the two data points $x$ and $y$. We shall consider some basic properties of affinity which will be sufficient to obtain the Newtonian  potential type structure for $K$.
\textbf{Symmetry}; affinity is a symmetric relation on $X\times X$ ($K(x,y)=K(y,x)$ for every $x$, $y\in X$).
\textbf{Positivity}; there is positive affinity between any couple of data points $x$ and $y$ ($K(x,y)\geq 0$ for every $x$, $y\in X$).
\textbf{Diagonal singularity}; the self affinity is unimprovable. Precisely, the affinity of each data point $x$ with itself is $+\infty$ ($K(x,x)=+\infty$ for every $x$) but for $x\neq y$ the affinity is finite ($K(x,y)<\infty$ for $x\neq y$).
\textbf{Quantitative Transiti\-vity}; if the affinity between the data points $x$ and $y$ is larger than $\lambda>0$ and the affinity between $y$ and $z$ is larger than $\lambda$ then the affinity between the points $x$ and $z$ is larger than $\nu(\lambda)<\lambda$. Here $\nu$ is a nonnegative, concave, increasing and continuous function defined on $\mathbb{R}^+$ onto  $\mathbb{R}^+$ such that $\nu(\lambda)<\lambda$.

A quasi-metric in $X$ is a nonnegative symmetric function $d$ defined on $X\times X$ which vanishes only on the diagonal $\Delta$ of $X\times X$ and satisfies a weak form of the triangle inequality, there exists a constant $\tau $ ($\geq 1$) such that the inequality
$d(x,z)\leq\tau(d(x,y)+d(y,z))$
holds for every $x$, $y$ and $z$ in $X$.

The main result of this paper can be stated as follows.

\begin{theorem}\label{thm:maintheorem}
Let $X$ be a set. Let $K: X\times X\to \mathbb{R}^+$ be a symmetric function satisfying the singularity and the quantitative transitivity conditions. Then there exist a decreasing and continuous function $\varphi$ defined in $\mathbb{R}^+$ and a quasi-metric $d$ on $X$ such that
\begin{equation*}
K(x,y)=\varphi(d(x,y)).
\end{equation*}
Moreover, $d(x,y)=h(x,y)\rho(x,y)$ with $\rho$ a metric on $X$ and $h$ a symmetric function such that for some constants $0<c_1<c_2<\infty$ satisfies $c_1\leq h(x,y)\leq c_2$ for every $x$ and $y$ in $X$.
\end{theorem}

The deepest results on the structure of quasi-metrics are due to Mac\'{\i}as and Segovia and are contained in \cite{MaSe79Lip}. See also \cite{AiIafNi98}. The most significant for our purposes is the fact that each quasi-metric is equivalent to a power of a metric. In other words, given a quasi-metric $d$ on $X$ with constant $\tau$, there exist $\beta\geq 1$ and a metric $\rho$ on $X$ such that for some positive constants $a_1$ and $a_2$ the inequalities
\begin{equation*}
a_1d(x,y)\leq\rho^{\beta}(x,y)\leq a_2d(x,y)
\end{equation*}
hold for every $x$ and $y$ in $X$. Actually the proof is based in Frink's lemma of metrization of uniform structures with a countable basis (\cite{Frink37}, \cite{Kelleybook75}).

The  rest of the paper is organized in the following way. Section~2 gives a characterization of quasi-metrics on a set in terms of properties of the family of stripes in $X\times X$ induced by the quasi-metric. Section~3 contains the construction of the monotonic profile. In Section~4 we prove the main result.

\section{Quasi-metrics and families of stripes around the diagonal}
Let $X$ be a set.
The composition of two subsets $U$ and $V$ of $X\times X$ is given by $
V\circ U=\{(x,z)\in X\times X: \textrm{ there exists }  y\in X \textrm{ such that } (x,y)\in U \textrm{ and } (y,z)\in V\}$.
A subset $U$ in $X\times X$ is said to be symmetric if $(x,y)\in U$ if and only if $(y,x)\in U$. Set $\Delta$ to
denote the diagonal in $X\times X$, i.e. $\Delta=\set{(x,x): x\in X}$. When a quasi-metric $\delta$ with constant $\tau$ is given
in $X$, it is easy to check that the one parameter family $\mathcal{V}(r)=\{(x,y)\in X\times X: \delta(x,y)<r\}$; $r>0$, of stripes around
 the diagonal of $X\times X$, satisfies
\begin{enumerate}[(S1)]
\item  each $\mathcal{V}(r)$ is symmetric;
\item $\Delta\subseteq\mathcal{V}(r)$, for every $r>0$;
\item $\mathcal{V}(r_1)\subseteq\mathcal{V}(r_2)$, for $0<r_1\leq r_2$;
\item $\cup_{r>0}\mathcal{V}(r)=X\times X$;
\item $\cap_{r>0}\mathcal{V}(r)\subseteq\Delta$;
\item there exists $T\geq 1$ such that $\mathcal{V}(r)\circ\mathcal{V}(r)\subseteq\mathcal{V}(Tr)$, for every $r>0$.
\end{enumerate}
Actually, the constant $T$ in (S6) can be taken to be the triangle constant $\tau$ of $\delta$. Set $\mathcal{P}(X\times X)$
to denote the set of subsets of $X\times X$.
\begin{theorem}\label{thm:Vdeltaequivalence}
Let $\mathcal{V}:\mathbb{R}^+\to \mathcal{P}(X\times X)$ be a one parameter family of the subsets of $X\times X$ that satisfies
\textit{(S1)} to \textit{(S6)} above. Then the function $\delta$ defined on $X\times X$ by
$\delta(x,y)=\inf\{r>0: (x,y)\in\mathcal{V}(r)\}$
is a quasi-metric on $X$ with $\tau\leq T$. Moreover, for each $\gamma>0$, we have
\begin{equation}\label{eq:Vdeltaequivalence}
\mathcal{V}_\delta(r)\subseteq\mathcal{V}(r)\subseteq\mathcal{V}_\delta((1+\gamma)r)
\end{equation}
hold for every $r>0$, where $\mathcal{V}_\delta(s)=\{(x,y)\in X\times X: \delta(x,y)<s\}$.
\end{theorem}
\begin{proof}
From \textit{(S4)} for the family $\mathcal{V}$ we see that $\delta(x,y)$ is well defined
as a nonnegative real number. The symmetry of $\delta$ follows from \textit{(S1)}. The fact that
$\delta$ vanishes on the diagonal $\Delta$ follows from $\Delta\subseteq\cap_{r>0}\mathcal{V}(r)$
which is contained in \textit{(S2)}. Now, if $(x,y)\in X\times X$ and $\delta(x,y)=0$, then, from \textit{(S3)}
for each $r>0$, $(x,y)\in\mathcal{V}(r)$. Now, from \textit{(S5)} we necessarily have that $(x,y)\in\Delta$ or,
in other words $x=y$. Let us check that $\delta$ satisfies a triangle inequality. Let $x$, $y$ and $z$ be three
points in $X$. Let $\varepsilon>0$. Take $r_1>0$ and $r_2>0$ such that $r_1<\delta(x,y)+\varepsilon$, $r_2<\delta(y,z)+\varepsilon$,
$(x,y)\in\mathcal{V}(r_1)$ and $(y,z)\in\mathcal{V}(r_2)$. From \textit{(S6)} with $r^*=\sup\{r_1,r_2\}$, we have
$(x,z)\in\mathcal{V}(r_2)\circ\mathcal{V}(r_1)\subseteq\mathcal{V}(r^*)\circ\mathcal{V}(r^*)\subseteq\mathcal{V}(Tr^*)$. Hence
$\delta(x,z)\leq Tr^*\leq T(r_1+r_2)\leq T(\delta(x,y)+\delta(y,z))+2\varepsilon T$ and we get the triangle inequality with $\tau=T$.
Notice first that from \textit{(S3)}, $\mathcal{V}_\delta(r)\subseteq\mathcal{V}(r)$ for every $r>0$. Take now $(x,y)\in \mathcal{V}(r)$,
then $\delta(x,y)\leq r<(1+\gamma)r$, so that $\mathcal{V}(r)\subseteq\mathcal{V}_\delta((1+\gamma)r)$ for every $\gamma>0$ and every $r>0$.
\end{proof}

The next result follows from the above and the metrization theorem of quasi-metric spaces proved in \cite{MaSe79Lip} as an application
of Frink's Lemma on metrizability of uniform spaces with countable bases.

\begin{theorem}\label{thm:stripesequivalent}
Let $\mathcal{V}$ and $\delta$ be as in Theorem~\ref{thm:Vdeltaequivalence}. Then, there exist a constant $\beta\geq 1$ and a metric $\rho$ on $X$ such
that
\begin{enumerate}[(i)]
\item  $4^{-\beta}\delta\leq\rho^\beta\leq 2^\beta\delta$;
\item $\mathcal{V}_{\rho^\beta}\left(\tfrac{r}{4^\beta}\right)\subseteq\mathcal{V}(r)\subseteq\mathcal{V}_{\rho^\beta}(2^{\beta +1}r)$ where
$\mathcal{V}_{\rho^\beta}(r)=\{(x,y)\in X\times X:\rho^\beta(x,y)<r\}$.
\end{enumerate}
\end{theorem}
\begin{proof}
Following the proof of Theorem~2 on page 261 in \cite{MaSe79Lip} take $\alpha<1$ such that $(3T^2)^\alpha=2$ and $\beta=\tfrac{1}{\alpha}>1$.
With $\rho$ the metric provided by the metrization theorem for uniform spaces with countable bases, we have that
$4\rho(x,y)\geq\delta(x,y)^{\alpha}\geq\tfrac{1}{2}\rho(x,y)$. So that
\begin{equation*}
\frac{1}{4^\beta}\delta(x,y)\leq\rho(x,y)^\beta\leq 2^\beta\delta(x,y),
\end{equation*}
and \textit{(ii)} follow from these inequalities and \eqref{eq:Vdeltaequivalence} with $\gamma=1$.
\end{proof}

\section{Building the basic profile shape}
The classical inverse proportionality to the square of the distance between the two bodies for the gravitation field,
translates into inverse proportionality  to the distance for the potential. That is $\varphi(r)=\tfrac{1}{r}$ for the
gravitational potential.

This section  is devoted to the construction of the basic shapes of the profiles that allow the Newtonian representation
of the kernels. This construction requires to solve a functional inequality involving the function $\nu$ that controls the
quantitative transitive property of $K$.

\begin{proposition}\label{propo:inequalityPsi}
Let $\nu$ be a concave, continuous, nonnegative and increasing function defined on $\mathbb{R}^+$ onto $\mathbb{R}^+$
such that $\nu(\lambda)<\lambda$ for every $\lambda>0$. Then, given $M>1$, there exists a continuous, decreasing
and convex function $\psi$ defined on $\mathbb{R}^+$ with $\psi(1)=1$ such that the inequality
\begin{equation}\label{eq:inequalityPsi}
\psi(\nu(\lambda))\leq M\psi(\lambda)
\end{equation}
holds for every $\lambda>0$.
\end{proposition}
\begin{proof}
Set $\lambda_0=1$, $\lambda_1=\nu(1)$ and $\lambda_{-1}=\nu^{-1}(1)$. Notice that $\lambda_1<1$ and $\lambda_{-1}>1$.
In fact, $\lambda_1=\nu(1)<1$ and $1=\nu^{-1}(\nu(1))=\nu^{-1}(\lambda_1)<\nu^{-1}(1)=\lambda_{-1}$. Set for $k\in \mathbb{N}$,
$\lambda_k=\nu(\lambda_{k-1})$ and $\lambda_{-k}=\nu^{-1}(\lambda_{-k+1})$. Notice that $\lambda_k$ decreases as $k\to\infty$ and
increases when $k\to-\infty$. The continuity of $\nu$ and the property $\nu(\lambda)<\lambda$ for every positive $\lambda$ imply
that $\lambda_k\to 0$ as $k\to\infty$ and $\lambda_k\to\infty$ as $k\to-\infty$ monotonically. This basic sequence $\{\lambda_k:k\in \mathbb{Z}\}$
allows to construct a function $\psi$ in the following way. Set $\psi(\lambda_k)=M^k$, $k\in \mathbb{Z}$ and for $\lambda\in [\lambda_{k+1},\lambda_k]$
define $\psi(\lambda)$ by linear interpolation. It is clear that $\psi$ is continuous, decreasing, $\psi(0^+)=+\infty$, $\psi(\infty)=0$,
$\psi(1)=\psi(\lambda_0)=M^0=1$ and that $\psi$ is convex. We only have to check that $\psi$ solves inequality \eqref{eq:inequalityPsi}. On the
sequence $\{\lambda_k: k\in \mathbb{Z}\}$, \eqref{eq:inequalityPsi} becomes an equality. In fact,
$\psi(\nu(\lambda_k))=\psi(\lambda_{k+1})=M^{k+1}=MM^k=M\psi(\lambda_k)$.

Let us now take $\lambda\in (\lambda_{k+1},\lambda_k)$ for $k\in \mathbb{Z}$. For such $\lambda$, $\psi(\lambda)$ satisfies
\begin{equation}\label{eq:trianglesemejant}
\frac{M^{k+1}-M^k}{\lambda_k-\lambda_{k+1}}=\frac{M^{k+1}-\psi(\lambda)}{\lambda-\lambda_{k+1}}.
\end{equation}
On the other hand, since $\lambda_{k+1}<\lambda<\lambda_k$, we have that $\lambda_{k+2}=\nu(\lambda_{k+1})<\nu(\lambda)<\nu(\lambda_k)=\lambda_{k+1}$.
Hence $\psi(\nu(\lambda))$ satisfies
\begin{equation}\label{eq:trianglesemejantwithnu}
\frac{M^{k+2}-M^{k+1}}{\lambda_{k+1}-\lambda_{k+2}}=\frac{M^{k+2}-\psi(\nu(\lambda))}{\nu(\lambda)-\lambda_{k+2}}.
\end{equation}
From \eqref{eq:trianglesemejant} and \eqref{eq:trianglesemejantwithnu} we get
\begin{equation*}
\psi(\lambda)=M^{k+1}-M^{k}(M-1)\frac{\lambda-\lambda_{k+1}}{\lambda_k-\lambda_{k+1}}
\end{equation*}
and
\begin{equation*}
\psi(\nu(\lambda))=M\left(M^{k+1}-M^{k}(M-1)\frac{\nu(\lambda)-\lambda_{k+2}}{\lambda_{k+1}-\lambda_{k+2}}\right).
\end{equation*}
Now, since $\nu$ is concave, we have for $\lambda_{k+1}<\lambda<\lambda_k$ that
\begin{equation*}
\frac{\nu(\lambda)-\nu(\lambda_{k+1})}{\lambda- \lambda_{k+1}}\geq\frac{\nu(\lambda_k)-\nu(\lambda_{k+1})}{\lambda_k-\lambda_{k+1}},
\end{equation*}
hence, $\psi(\nu(\lambda))\leq M\psi(\lambda)$.
\end{proof}

The basic shapes for the profiles $\varphi$ in our main result will be given as composition of the inverse $\eta$ of $\psi$ with power laws.

\section{Proof of the main result}
Let us start by rewriting, formally, the properties of symmetry, positivity, singularity and transitivity of a data affinity kernel $K(x,y)$
 defined on the set $X\times X$. Let $K: X\times X\to \mathbb{R}$ such that
\begin{enumerate}[(K1)]
\item $K(x,y)=K(y,x)$, for every $x$ and $y$ in $X$;
\item $K(x,y)>0$, for every $x$ and $y$ in $X$;
\item $K(x,y)=+\infty$ if and only if $x=y$;
\item there exists a continuous, concave, increasing and nonnegative function $\nu$ defined on $\mathbb{R}^+$ onto $\mathbb{R}^+$, with
$\nu(\lambda)<\lambda$, $\lambda>0$, such that $K(x,z)>\nu(\lambda)$ whenever there exists $y\in X$ with $K(x,y)>\lambda$ and $K(y,z)>\lambda$,
for every $\lambda >0$.
\end{enumerate}
With these properties, Theorem~\ref{thm:maintheorem} can be restated as follows.
\begin{theorem}
Let $X$ be a set. Let $K$ be a kernel on $X\times X$ satisfying properties\textit{(K1)} to \textit{(K4)}.
Then, there exist a metric $\rho$ on $X$, a real number $\beta\geq 1$, a function $h(x,y)$ defined on $X\times X$ with  $2^{-1/\beta}\leq h(x,y)\leq 4$ and a function $\varphi:\mathbb{R}^+\to\mathbb{R}^+$ continuous, decreasing with $\varphi(0^+)=+\infty$ and $\varphi(\infty)=0$, such that
\begin{equation*}
K(x,y)=\varphi(h(x,y)\rho(x,y)).
\end{equation*}
\end{theorem}
\begin{proof}
Let $\nu$ be the function provided by (K4). Let  $\psi$ be given by Proposition~\ref{propo:inequalityPsi} with this function
$\nu$, and $M=2$. Hence $\psi(\nu(\lambda))\leq 2\psi(\lambda)$ for every $\lambda>0$. Take $\eta=\psi^{-1}$ and
$\mathcal{V}: \mathbb{R}^+\to\mathcal{P}(X\times X)$ given by
\begin{equation*}
\mathcal{V}(r)=E_{\eta(r)}=\{(x,y): K(x,y)>\eta(r)\}.
\end{equation*}
Let is check that $\mathcal{V}$ satisfies properties (S1) to (S6) in Section~2 with  constant $T=2$ ($=M$). From (K1) we see
that each $E_{\lambda}$ is symmetric, in particular $\mathcal{V}(r)$ is symmetric for every $r>0$. Since $K(x,x)=+\infty$, from
(K3), we have that $\Delta\subseteq E_{\eta(r)}=\mathcal{V}(r)$ for $r>0$. In order to check (S3) take $0<r_1<r_2<\infty$.
Hence $\eta(r_1)>\eta(r_2)$, so that $K(x,y)>\eta(r_1)$ implies  $K(x,y)>\eta(r_2)$. Or, in other words $E_{\eta(r_1)}\subset E_{\eta(r_2)}$.
Or $\mathcal{V}(r_1)\subseteq\mathcal{V}(r_2)$. The positivity (K2) of $K(x,y)$ shows (S4). Property (S5) of $\mathcal{V}$ follows
from (S3). To prove (S6) for $\mathcal{V}$, take $r>0$. If $(x,z)\in\mathcal{V}(r)\circ\mathcal{V}(r)=E_{\eta(r)}\circ E_{\eta(r)}$, then
there exists $y\in X$ such that $K(x,y)>\eta(r)$ and $K(y,z)>\eta(r)$. From (K4), $K(x,z)>\nu(\eta(r))$. Now applying \eqref{eq:inequalityPsi}
with $\lambda=\psi^{-1}(r)$ we get $K(x,z)>\nu(\psi^{-1}(r))\geq\eta(2r)$, or
$(x,z)\in E_{\eta(2r)}=\mathcal{V}(2r)$. Hence (S6) for $\mathcal{V}$ holds with $T=2$. We can, then apply the results of Section~2. First
to produce a quasi-metric $\delta$ as in Theo\-rem~\ref{thm:Vdeltaequivalence} and then to provide the metric $\rho$ and the exponent $\beta$
given in Theo\-rem~\ref{thm:stripesequivalent}. Thus, for every $r>0$, $\mathcal{V}_{\rho^\beta}\left(\tfrac{r}{4^\beta}\right)\subseteq\mathcal{V}(r)=E_{\eta(r)}\subseteq\mathcal{V}_{\rho^\beta}(2^{\beta +1}r)$,
where $\rho$ is a metric in $X$ and, since $T$ can be taken to be equal to 2, $\beta=\log_2 12$ works. The above inclusiones, taking
$s=r^{1/\beta}$, are equivalent to
\begin{equation}\label{eq:inclusionsforrho}
\left\{\rho<\tfrac{s}{4}\right\}\subseteq \left\{(\psi\circ K)^{1/\beta}>s\right\}\subseteq\left\{\rho<2^{1+1/\beta}s\right\}
\end{equation}
for every $s>0$. Let $x$ and $y$ be two different points in $X$. Since $0<K(x,y)<\infty$ so is $(\psi\circ K)^{1/\beta}(x,y)$.
There exists, then, a unique $j\in \mathbb{Z}$ ($j=j(x,y)$) such that $2^j\leq(\psi\circ K)^{1/\beta}(x,y)<2^{j+1}$.
By the second inclusion in \eqref{eq:inclusionsforrho} we see that
$\rho(x,y)<2^{1/\beta}2^{j}\leq 2^{1/\beta}(\psi\circ K)^{1/\beta}(x,y)$. On the other hand, since $(\psi\circ K)^{1/\beta}(x,y)<2^{j+1}$, from
 the first inclusion in \eqref{eq:inclusionsforrho} we necessarily have that $\rho(x,y)\geq\tfrac{2^{j+1}}{4}>\tfrac{1}{4}(\psi\circ K)^{1/\beta}(x,y)$. Hence for $x\neq y$ we have
\begin{equation*}
\frac{1}{4}(\psi\circ K)^{1/\beta}\leq\rho\leq 2^{1/\beta}(\psi\circ K)^{1/\beta}.
\end{equation*}
Set $h(x,y)=\tfrac{(\psi\circ K)^{1/\beta}}{\rho(x,y)}$ for $x\neq y$ and $h(x,x)=1$. Then $\tfrac{1}{2^{1/\beta}}\leq h\leq 4$ and
$K(x,y)=\psi^{-1}((h(x,y)\rho(x,y))^\beta)=\varphi(h(x,y)\rho(x,y))$ with $\varphi(r)=\psi^{-1}(r^\beta)=\eta(r^\beta)$.
\end{proof}

Notice that since $h$ is symmetric and bounded above and below by posi\-tive constants the function $d(x,y)=h(x,y)\rho(x,y)$ is
a quasi-metric. But actually $d$ is better than a general quasi-metric since its triangular inequality constant can be taken
to be independent of the length of chains. Precisely, $d(x_1,x_m)\leq 2^{2+1/\beta}\sum_{j=1}^{m-1}d(x_j,x_{j+1})$.

Let us observe also that Newtonian type power laws are obtained when $\nu(\lambda)=a\lambda$ for $a<1$. In fact, with
$m=\tfrac{1}{\log_2 a}<0$, $\psi(r)=r^m$ solves the equation $(a\lambda)^m=2\lambda^m$. Hence $\varphi$ becomes also a
power law.



\bigskip

%
\noindent{\footnotesize
\textsc{Instituto de Matem\'atica Aplicada del Litoral (IMAL), UNL, CONICET, FIQ.}

\smallskip
\noindent\textmd{CCT-CONICET-Santa Fe, Predio ``Dr. Alberto Cassano'', Colectora Ruta Nac.~168 km 0, Paraje El Pozo, 3000 Santa Fe, Argentina.}
}
\bigskip

\end{document}